\documentclass[12pt]{amsart}

\usepackage{amssymb, mathrsfs}
\usepackage{graphicx}

\usepackage{lineno}

\usepackage{color}

\newtheorem{theorem}{Theorem}

\newtheorem{proposition}[theorem]{Proposition}
\newtheorem{observation}[theorem]{Observation}
\newtheorem{corollary}[theorem]{Corollary}

\newtheorem{problem}{Problem}

\newtheorem{them}{Theorem}
\newtheorem{lema}[them]{Lemma}

\newtheorem{observ}[them]{Observation}

\newtheorem{prop}[them]{Proposition}

\newtheorem{example}[theorem]{Example}

\newtheorem{remark}[theorem]{Remark}

\begin{document}

\title{Roman domination: changing, unchanging, $\gamma_R$-graphs }

\author[]{Vladimir Samodivkin}
\address{Department of Mathematics, UACEG, Sofia, Bulgaria}
\email{vl.samodivkin@gmail.com}
\today

\begin{abstract}
A Roman dominating function (RD-function) on a graph $G = (V(G), E(G))$
 is a labeling $f : V(G) \rightarrow \{0, 1, 2\}$ such
that every vertex with label $0$ has a neighbor with label $2$. 
The weight $f(V(G))$ of a RD-function $f$ on $G$ is the value $\Sigma_{v\in V(G)} f (v)$. 
The {\em Roman domination number} $\gamma_{R}(G)$ of $G$ 
 is the minimum weight of a RD-function on $G$. 
 The six classes of graphs resulting from the changing or unchanging 
of the Roman domination number of a graph when a vertex is deleted,
 or an edge is deleted or added are considered. We consider relationships among the
classes, which are illustrated in a Venn diagram.
A graph $G$ is Roman domination $k$-critical if the removal of any
set of $k$ vertices decreases the Roman domination number.
Some initial properties of these graphs are studied. 
 The $\gamma_R$-graph of a graph $G$,  is any graph which  vertex set
is the collection $\mathscr{D}_R(G)$ of all minimum weight RD-functions on $G$. 
We define adjacency between any two  elements of $\mathscr{D}_R(G)$  
in several ways,  and  initiate the study  of the obtained  $\gamma_R$-graphs.
\end{abstract}

\maketitle


\section{Introduction and preliminaries } 
By a graph, we mean a finite, undirected graph with neither loops nor multiple edges.
For basic notation and graph theory terminology not explicitly defined here, we
in general follow Haynes et al. \cite{hhs1}.
We denote the vertex set and the edge set of a graph $G$ by $V(G)$ and $ E(G),$  respectively. 
A {\em spanning subgraph} for $G$ is a subgraph of $G$ which contains every vertex of $G$. 
In a graph $G$, for a subset $S \subseteq V (G)$ the {\em subgraph induced} by $S$ is the graph
$\left\langle S \right\rangle$ with vertex set $S$ and edge set $\{xy \in E(G) \mid x, y \in S\}$.
The complement $\overline{G}$ of $G$ is the graph whose
vertex set is $V (G)$ and whose edges are the pairs of nonadjacent vertices of $G$. 
The notation $G \simeq H$ will be used to denote that G and H are isomorphic. 
We write (a) $K_n$ for the {\em complete graph} of order $n$, 
(b) $K_{m,n}$ for the {\em complete bipartite graph} with partite sets of order $m$ and $n$, and 
(c) $P_n$ for the  {\em path} on $n$ vertrices. 
Let $C_m$ denote the {\em cycle} of length $m$. 
 For vertices $x$ and $y$ in a connected graph $G$, 
the {\em distance} $dist(x, y)$ is the length of a shortest $x-y$ path in $G$.
	For any vertex $x$ of a graph $G$,  $N_G(x)$ denotes the set of all  neighbors of $x$ in $G$,  
	$N_G[x] = N_G(x) \cup \{x\}$ and the degree of $x$ is $deg(x,G) = |N_G(x)|$. 
The {\em minimum} and {\em maximum} degrees
 of a graph $G$ are denoted by $\delta(G)$ and $\Delta(G)$, respectively.
   For a graph $G$, let $x \in X \subseteq V(G)$. 
	  A vertex $y \in V(G)$ is an $X$-{\em private neighbor}  of $x$ if $N_G[y] \cap X = \{x\}$. 
		The set of all $X$-private neighbors of $x$ is denoted by $pn_G[x,X]$.
		 A {\em leaf} of a graph is a vertex of degree $1$, while a {\em support vertex} is a vertex adjacent to a leaf.
 A {\em vertex cover} of a graph is a set of vertices such that 
each edge of the graph is incident to at least one vertex of the set.

   The study of domination and related subset problems is one of the fastest growing areas in graph theory.
For a comprehensive introduction to the theory of domination in graphs we refer the reader
 to Haynes et al. \cite{hhs1}.
A {\em dominating set} for a graph $G$ is a subset $D\subseteq V(G)$ of 
vertices such that every vertex not in $D$ is adjacent to
at least one vertex in $D$. The minimum cardinality of a dominating 
set is called the {\em domination number} of $G$ and is denoted by $\gamma (G)$.

A variation of domination called Roman domination was introduced by ReVelle \cite{re1,re2}.
 Also see ReVelle and Rosing \cite{rer} for an integer programming
formulation of the problem. 
The concept of Roman domination can be formulated in terms of graphs  (\cite{cdhh}).
A {\em Roman dominating function} ({RD-{\em function}) on a
graph $G$ is a vertex labeling $f : V(G) \rightarrow \{0, 1, 2\}$
 such that every vertex with label $0$ has a neighbor with label $2$. 
 For a RD-function $f$, let $ V_i^f = \{v \in V (G) : f(v) = i\}$ for i = 0, 1, 2. 
Since these $3$ sets determine $f$, we can equivalently write  $f=(V_0^f; V_1^f; V_2^f)$.
 The {\em weight} $f(V(G))$ of a RD-function $f$ on $G$ is the value $\Sigma_{v\in V(G)} f (v)$, 
  which equals $|V_1^f| + 2|V_2^f|$. 
The {\em Roman domination number} $\gamma_{R}(G)$ of $G$ 
 is the minimum weight of a RD-function on $G$. 
A  RD-function  with minimum weight in a graph $G$
 will be referred to as a $\gamma_{R}$-{\em function} on $G$. 
Denote by $\mathscr{D}_R(G)$ the set of all $\gamma_R$-functions on $G$ 
and $\# \gamma_R (G) = |\mathscr{D}_R(G)|$. 
If $H$ is a subgraph of $G$ and $f$ a $\gamma_R$-function on $G$, 
then we denote the  restriction of $f$ on $H$ by $f|H$.

It is often of interest to known how the value of a graph parameter $\mu$ is affected 
when a change is made in a graph. 
The addition of a set of edges,  or the removal of a set of  vertices/edges may increase or decrease $\mu$, or leave $\mu$ unchanged.  
Thus, it is naturally to consider the following classes of graphs. 
We use acronyms to denote these classes ($V$ represents vertex; $E$: edge; $R$: removal; $A$: addition).  
Let $k$ be a positive integer.

\begin{itemize}
\item[(i)] $(k$-$VR^-_{\mu})$   \hspace{.5cm} $\mu(G-S) < \mu (G)$  for any set  $S \subseteq  V(G)$ with $|S|=k$,
\item[(ii)] $(k$-$VR^+_{\mu})$  \hspace{.5cm} $\mu(G-S) > \mu (G)$  for any set  $S \subseteq  V(G)$ with $|S|=k$,			
\item[(iii)] $(k$-$VR^=_{\mu})$  \hspace{.5cm} $\mu(G-S) = \mu (G)$ for any set  $S \subseteq  V(G)$ with $|S|=k$,		
\item[(iv)] $(k$-$VR^{\not=}_{\mu})$  \hspace{.5cm} $\mu(G-S) \not= \mu (G)$ for any set  $S \subseteq  V(G)$ with $|S|=k$															
\item[(v)] $(k$-$ER^-_{\mu})$   \hspace{.5cm} $\mu(G-R) < \mu (G)$  for any set  $R \subseteq  E(G)$ with $|R|=k$,								
\item[(vi)] $(k$-$ER^+_{\mu})$   \hspace{.5cm} $\mu(G-R) > \mu (G)$ for any set  $R \subseteq  E(G)$ with $|R|=k$,
\item[(vii)] $(k$-$ER^=_{\mu})$   \hspace{.5cm} $\mu(G-R) = \mu (G)$  for any set  $R \subseteq  E(G)$ with $|R|=k$,	
\item[(viii)] $(k$-$ER^{\not=}_{\mu})$   \hspace{.5cm} $\mu(G-R) \not= \mu (G)$  for any set  $R \subseteq  E(G)$ with $|R|=k$,	
\item[(ix)] $(k$-$EA^-_{\mu})$  \hspace{.5cm}  $\mu(G+U) < \mu (G)$  for any set  $U \subseteq  E(\overline{G})$ with $|U|=k$,	
\item[(x)] $(k$-$EA^+_{\mu})$  \hspace{.5cm}  $\mu(G+U) > \mu (G)$  for any set  $U \subseteq  E(\overline{G})$ with $|U|=k$,	
\item[(xi)] $(k$-$EA^=_{\mu})$  \hspace{.5cm} $\mu(G+U) = \mu (G)$ for any set  $U \subseteq  E(\overline{G})$ with $|U|=k$,			
\item[(xii)] 	$(k$-$EA^{\not=}_{\mu})$  \hspace{.5cm} $\mu(G+U) \not= \mu (G)$ 
for any set  $U \subseteq  E(\overline{G})$ with $|U|=k$.
\end{itemize}

Two mathematical problems arise immediately: 
1) to find a nontrivial characterization of  every of the above classes, and 
	2) to establish relationships among these twelve classes. 
In Section \ref{six}, we concentrate on the second problem 
in the case when $\mu \equiv \gamma_R$ and $k=1$. 
In Section \ref{cvrk}  we present some initial 
 results on the class $k$-$VR^-_{\gamma_R}$.  
This class was introduced by Jafari Rad in \cite{r}.
The $\gamma_R$-graph of a graph $G$  is the graph which  vertex set is $\mathscr{D}_R(G)$.
In Section \ref{gammaR} we define adjacency between any two 
 $\gamma_R$-functions  of a graph in several ways,  
and  initiate the study  of the obtained  $\gamma_R$-graphs.

We end this section with some known results which will be useful in proving our main results.

\begin{observ}[\cite{cdhh}]  \label{o2}   
Let $f = (V_0^f; V_1^f; V_2^f)$ be any $\gamma_{R}$-function on a graph $G$. 
Then $\Delta (\left\langle V_1^f \right\rangle) \leq 1$ and 
no edge of $G$ joins $V_1^f$ and $V_2^f$.  If  $|V_1^f|$ is a minimum then $V_1^f$ 
is independent and if in addition $G$ is isolate-free then $V_0^f \cup V_2^f$ is a vertex cover. 
\end{observ}

In most cases,  Observation \ref{o2}   will be used in the sequel without specific reference.

\begin{them}[\cite{rv1}]\label{minus1}
Let $v$ be a vertex of a graph $G$. Then $\gamma_R(G-v) < \gamma_R(G)$ 
 if and only if there is a $\gamma_R$-function $f$ on $G$ such that $v \in V^f_1$.  
If $\gamma_R(G-v) < \gamma_R(G)$  then $\gamma_R(G-v) = \gamma_R(G) - 1$. 
If $\gamma_R(G-v) > \gamma_R(G)$  then  for every  $\gamma_{R}$-function  $f$  on $G$, $f(v) = 2$.
\end{them}

 According to the effects of vertex removal on the Roman domination number of a graph $G$, let \cite{rhv}

$\bullet$\ $V_{R}^+(G) = \{v \in V(G) \mid \gamma_{R} (G-v) > \gamma_{R} (G)\}$, 

$\bullet$\ $V_{R}^-(G) = \{v \in V(G) \mid \gamma_{R} (G-v) < \gamma_{R} (G)\}$,

$\bullet$\ $V_{R}^=(G) = \{v \in V(G) \mid \gamma_{R} (G-v) = \gamma_{R} (G)\}$.

Clearly $V_{R}^-(G), V_{R}^=(G)$ and $V_{R}^+(G)$ are paired disjoint,  and their union is  $V(G)$.

\begin{them}\label{addedge1} {\rm(\cite{rhv})} 
Let $x$ and $y$ be  non-adjacent vertices  of a graph $G$.
Then $\gamma_R(G) \geq \gamma_R(G+xy) \geq \gamma_R(G)-1$. 
Moreover, $\gamma_R(G+xy) = \gamma_R(G)-1$ if and only if  there is a $\gamma_R$-function 
$f$ on $G$  such that $\{f(x), f(y)\} = \{1, 2\}$.
\end{them}

\newpage

	\section{Six classes} \label{six}

We will write $\mathcal{R}_{CVR}$,  $\mathcal{R}_{UVR}$,   $\mathcal{R}_{CER}$, $\mathcal{R}_{UER}$, $\mathcal{R}_{CEA}$,
    and  $\mathcal{R}_{UEA}$ 
instead of 
$1$-$VR^-_{\gamma_R}$,  $1$-$VR^=_{\gamma_R}$, $1$-$ER^+_{\gamma_R}$, 
 $1$-$ER^=_{\gamma_R}$,  $1$-$EA^-_{\gamma_R}$, and $1$-$EA^=_{\gamma_R}$, respectively.
 The first four classes of graphs  were introduced in \cite{rv1}    by Jafari Rad and Volkmann. 
On the other hand, the graphs in $\mathcal{R}_{CEA}$ and  $\mathcal{R}_{UEA}$  
were investigated by Hansberg et al. \cite{rhv}, and Chellali and Jafari Rad \cite{r}, respectively. 
Let us note that Theorems \ref{minus1}  and  \ref{addedge1}  imply that 
(a) the class $1$-$VR^+_{\gamma_R}$ is empty, 
 (b) the class $1$-$EA^+_{\gamma_R}$ consists of all complete graphs,
(c)  the class  $1$-$ER^-_{\gamma_R}$ consists of all edgeless graphs, 
(d) $1$-$VR^{\not=}_{\gamma_R} \equiv  \mathcal{R}_{CVR}$, 
(e)  $1$-$ER^{\not=}_{\gamma_R} \equiv  \mathcal{R}_{CER}$, and 
(f) $1$-$EA^{\not=}_{\gamma_R} \equiv  \mathcal{R}_{CEA}$.
That is why we concentrate,  in what follows,  
on the establishing relationships among the following six classes:
$\mathcal{R}_{CVR}$,  $\mathcal{R}_{UVR}$,   $\mathcal{R}_{CER}$, $\mathcal{R}_{UER}$, $\mathcal{R}_{CEA}$,
    and  $\mathcal{R}_{UEA}$. 
For further results on these classes see  \cite{crv}, \cite{hr}, \cite{rhv0} and    \cite{samdmaa}.  
Our main goal in this section is to show that these six classes
are related as in the Venn diagram of Fig. \ref{fig:regions0}.

	\begin{figure}[htbp]
	\centering
		\includegraphics{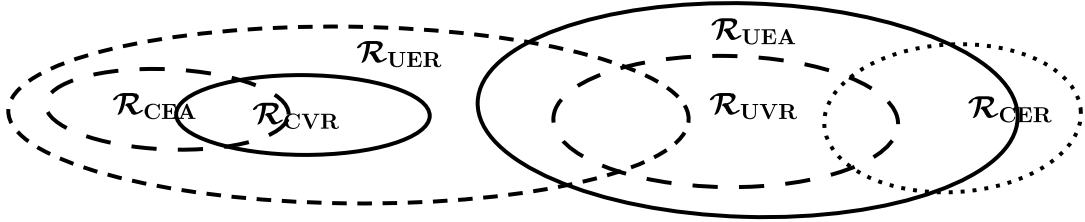}
	\caption{Classes of changing and unchanging graphs.}
	\label{fig:regions0}
\end{figure}

\begin{theorem}\label{ceainuer}
Let a graph $G$ be in $\mathcal{R}_{CEA}$. Then all the following hold. 
\begin{itemize}
\item[(i)] {\rm(\cite{crv})} $V (G) = V^-(G)\cup V^=(G)$ and either $V^=(G)$  
                        is empty or $\left\langle V^=(G) \right\rangle$ is a complete graph.
\item[(ii)] 	A vertex $x \in V^=(G)$ if and omly if there are $\gamma_R$-functions 
										$f_x$ and $g_x$ on $G$  with $\{f_x(x), g_x(x)\} = \{0,2\}$. 
\item[(iii)] If $V^=(G)$ is not empty and $\left\langle V^=(G)\right\rangle$  is not a connected component 
                     of $G$, then each vertex in $V^=(G)$  has a neighbor in $V^-(G)$. 
\item[(iv)]  $G$ is in $\mathcal{R}_{UER}$. 
\end{itemize}
\end{theorem}

\begin{proof}
 For  complete graphs the results are obvious. So, let $G$ be noncomplete. 

(ii)        By Theorem \ref{minus1}, $V^=(G) = A\cup B \cup C$, where 
			$A = \{x \in V(G) \mid  f(x) = 0 \mbox{ for each } \gamma_R-\mbox{function } f \mbox{ on } G\}$, 
			$B =  \{x \in V(G) \mid  \gamma_R(G-x) = \gamma_R(G) \mbox{ and } f(x) = 2 \mbox{ for each } \gamma_R-\mbox{function } f \mbox{ on } G \}$, and 
			$C = \{x \in V(G) \mid  \mbox{ there are }  \gamma_R-\mbox{functions } f_x \mbox{ and } g_x\mbox{ with } \{f_x(x), g_x(x)\} = \{0,2\}\}$. 
			
		 Theorem \ref{addedge1} implies that  $A$ is empty.
    Suppose $B$ is not empty, and  $u \in B$. 
		By (i) we have  $B \subseteq V^=(G) \subseteq N[u]$. 
		Now  Observation \ref{o2} and Theorem \ref{minus1} lead to  $N[u]=V^=(G)$ and  $B \subsetneq V^=(G)$. 
		 Since $A = \emptyset$,  there is $v \in C$. 
	 But then there exists a $\gamma_R$-function $f$ on $G$ with $f(v)=2$. 
		Define a RD-function $f^\prime$ on $G$ as follows: 	 $f^\prime(u)=0$  and $f^\prime(x) = f(x)$ 
		for all $x \in V(G-x)$. Since $f^\prime$ has a weight less than $\gamma_R(G)$,
		we arrive to a contradiction. 		Thus $V^=(G) = C$, as required. 
		\medskip
    
(iii) Assume to the contrary, that $N[v] = V^=(G)$ for some $v \in V^=(G)$. 
     Clearly, there are $u \in V^-(G)$ and $w \in V^=(G)$ which are adjacent.
      Since $uv \not\in E(G)$ and $G$ is in $\mathcal{R}_{CEA}$, there is 
      a $\gamma_R$-function $f^{\prime\prime}$ on $G$ with 
      $f^{\prime\prime}(u)=1$ and $f^{\prime\prime}(v)=2$.
      But then $f^{\prime\prime}(w)=0$ and $f^{\prime\prime\prime} = 
      ((V_0^{f^{\prime\prime}}(G)-\{w\})\cup \{u,v\}; V_1^{f^{\prime\prime}}-\{u\}; 
      (V_2^{f^{\prime\prime}}-\{v\}) \cup \{w\})$ is a RD-function on $G$ 
      with weight less than $\gamma_R(G)$, a contradiction.     
\medskip
      
(iv) Assume $G \in \mathcal{R}_{CEA}-\mathcal{R}_{UER}$. 
      Then there is an edge $x_1x_2 \in E(G)$ with $\gamma_R(G_{12}) > \gamma_R(G)$, 
      where $G_{12} = G-x_1x_2$. Now by Theorem \ref{addedge1}, applied to $G_{12}$ and $x_1x_2$, 
      there is a $\gamma_R$-function $f$ on $G_{12}$ with $\{f(x_1), f(x_2)\}= \{1,2\}$, 
      say without loss of generality, $f(x_1) = 2$. Note also that 
      $f_{12} = (V_0^f(G) \cup\{x_2\}; V_1^f(G)-\{x_2\}; V_2^f(G))$ is a $\gamma_R$-function 
      on $G$. Since $G$ is in $\mathcal{R}_{CEA}$, we already know that 
      $V(G) = V^ =(G) \cup V^-(G)$. If there is a $\gamma_R$-function 
      $f^\prime$ on $G$ with $f^\prime(x_i)=1$, then $f^\prime$ is a RD-function on $G_{12}$, 
      a contradiction. Thus, $x_1, x_2 \in V^=(G) = C$.          
      
      Suppose that $x_1 \in V^+(G_{12}) \cup V^=(G_{12})$. 
      Then $\gamma_R(G-x_1) = \gamma_R(G_{12}-x_1) \geq \gamma_R(G_{12}) > \gamma_R(G)$. 
      This immediately implies $x_1 \in V^+(G)$, a contradiction. 
      
      So, in what follows let $x_1 \in V^-(G_{12})$. 
      If $\left\langle V^=(G)\right\rangle$ is a component of $G$, then 
      $\gamma_R(G_{12}) = \gamma_R(G)$, a contradiction. 
      Hence each vertex in $V^=(G)$ is adjacent to a vertex in $V^-(G)$ (by (iii)). 
      Assume first that $y \in V^-(G)$ is adjacent to both $x_1$ and $x_2$. 
      Then there is a $\gamma_R$-function $g$ on $G$ with $g(y) = 1$. 
      This implies $g(x_1) = g(x_2) = 0$ (recall that $x_1,x_2 \in V^=(G)$). 
      But then $g$ is a RD-function on $G_{12}$ with weight less than $\gamma_R(G_{12})$, 
      a contradiction. Thus, all common neighbors of $x_1$ and $x_2$ are in $V^=(G)$. 
      Suppose $x_3 \in V^=(G)$ and $u \in N(x_1) \cap V^-(G)$. 
			If  $ux_3 \not \in E(G)$ then there is a $\gamma_R$-function $f_1$ on $G$ 
			with $f_1(x_3) = 2$ and $f_1(u)=1$.  Since $f_1$ is a RD-function on $G_{12}$, 
			we arrive to a contradiction. 
	    Therefore $N[x_1] = N[x_3]$, which implies $f(x_3) = 0$. 
			But then   $f_2 = (V_0^f-\{x_3\}\cup \{x_1,x_2\}; V_1^f-\{x_2\}; V_2^f-\{x_1\} \cup \{x_3\})$ 
			is a RD-function on $G_{12}$ of weight less than $\gamma_R(G_{12})$, 
			a contradiction.

      Thus, $V^=(G) = \{x_1,x_2\}$ and $N(x_1) \cap N(x_2) = \emptyset$. 
      Let $N(x_1)-\{x_2\} = \{y_1,y_2,..,y_r\}$ and $N(x_2)-\{x_1\} = \{z_1,z_2,..,z_s\}$. 
      If there are nonadjacent $y_i$ and $y_j$, then there is a $\gamma_R$-function
       $g$ on $G$ with $\{g(y_i), g(y_j)\} = \{1,2\}$. Hence $g(x_1) = 0$ which 
       implies that $g$ is a RD-function on $G_{12}$, a contradiction. Thus 
       $\left\langle N[x_i]-\{x_j\}\right\rangle$ is a complete graph for $\{i,j\} = \{1,2\}$.
       
       Assume now that $y_iz_j \not\in E(G)$. Then, without loss of generality, 
       there is a $\gamma_R$-function $l$ on $G$ with $l(y_i)= 2$ and $l(z_j)  =1$. 
      Since $x_2 \in V^=(G)$,  $l(x_2) = 0$. If $l(x_1) \not= 2$, then $l$ is a RD-function on $G_{12}$, 
       a contradiction. Thus $l(x_1)=2$. But then 
       $l_1  = (V_0^l(G)-\{x_2\}; V_1^l(G) \cup \{x_1,x_2\}; V_2^l(G)-\{x_1\})$ 
       is a $\gamma_R$-function on $G$ and $l_1(x_1) = l_1(x_2) = l_1(y_j) =1$, a contradiction.
       So, $(N(x_1)\cup N(x_2)) - \{x_1,x_2\}$ induce a complete graph. 
       Now, let $h$ be any $\gamma_R$-function on $G$ with $h(x_1) = 2$ and $h(z_1) = 1$. 
       But then  
       $h^\prime  = (V_0^h(G), (V_1^h(G) - \{z_1\}) \cup \{x_1\}; (V_2^h(G)-\{x_1\} \cup \{z_1\})$ 
       is a $\gamma_R$-function on $G$ with $h^\prime(x_1) = 1$, a contradiction.
\end{proof}

\begin{theorem}\label{cvruer}
Let $V^-(G)$ contain a vertex cover of a graph $G$. Then $G$ is in $\mathcal{R}_{UER}$.
 In particular, if  $G$ is in $\mathcal{R}_{CVR}$, then $G$ is in $\mathcal{R}_{UER}$.
\end{theorem}

\begin{proof}
Let $xy \in E(G)$.  Since $V^-(G)$ contains  a vertex cover of $G$, 
at least one of $x$ and $y$ is in $V^-(G)$, say $x \in V^-(G)$.  
Then there is a $\gamma_R$-function $f$ on $G$ with $f(x)=1$. 
This immediately implies that for each edge $e \in E(G)$ incident to $x$, 
 $f$ is a RD-function on $G-e$. Since always
 $\gamma_R(G-e) \geq \gamma_R(G)$, we obtain  $\gamma_R(G-e) = \gamma_R(G)$. 
Since $V^-(G)$ contains  a vertex cover of $G$, a graph $G$ is in $\mathcal{R}_{UER}$.  
The rest is obvious.
\end{proof}

\begin{lema}\cite{cr}\label{ueaknown}
Let $G$ be a graph of order $n \geq 3$. A graph $G$ is in $\mathcal{R}_{UEA}$  
 if and only if for every $\gamma_R$-function $f = (V_0, V_1, V_2)$, $V_1 = \emptyset$. 
\end{lema}

In order to establish a Venn diagram representing the classes 
$\mathcal{R}_{CVR}$,  $\mathcal{R}_{UVR}$,   $\mathcal{R}_{CER}$, $\mathcal{R}_{UER}$, $\mathcal{R}_{CEA}$,
    and  $\mathcal{R}_{UEA}$, we do not consider the cases that are  vacuously true. 
For example (a) the complete graphs are in both $\mathcal{R}_{CEA}$  and  $\mathcal{R}_{UEA}$, 
and (b)  the edgeless graphs are in both  $\mathcal{R}_{CER}$ and $\mathcal{R}_{UER}$.
Therefore we  exclude edgeless graphs and complete graphs. 

To continue, we need to relabel the Venn diagram of Fig.\ref{fig:regions0} 
in $11$ regions $R_1-R_{11}$ as shown in Fig. \ref{fig:regions}. 

	\begin{figure}[htbp]
	\centering
		\includegraphics{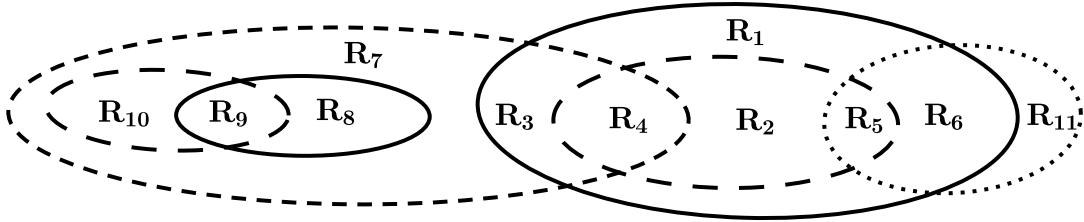}
	\caption{Regions of Venn diagram: general case}
	\label{fig:regions}
\end{figure}

\begin{theorem}\label{Rclassification}
Classes $\mathcal{R}_{CVR}$, $\mathcal{R}_{CEA}$, $\mathcal{R}_{CER}$, $\mathcal{R}_{UVR}$, 
$\mathcal{R}_{UER}$ and $\mathcal{R}_{UEA}$  
are related as shown in the  Venn diagram of Fig. \ref{fig:regions0}.
\end{theorem}

\begin{proof}
By Theorems \ref{ceainuer}  and  \ref{cvruer} we have 
$\mathcal{R}_{CEA} \cup \mathcal{R}_{CVR} \subseteq \mathcal{R}_{UER}$. 
It is obvious that all  $\mathcal{R}_{UER} \cap \mathcal{R}_{CER}$, 
$\mathcal{R}_{UVR} \cap \mathcal{R}_{CVR}$, and $\mathcal{R}_{UEA} \cap \mathcal{R}_{CEA}$ 
are empty. 
If a graph $G$ is in $\mathcal{R}_{UVR}$, then clearly $V(G) = V^=(G)$.  
Lemma \ref{ueaknown} now implies $\mathcal{R}_{UVR} \subseteq \mathcal{R}_{UEA}$. 
If $G \in \mathcal{R}_{CVR}$ then $V^-(G) \not= \emptyset$ and by Lemma \ref{ueaknown}, 
 $\mathcal{R}_{CVR}$ and $\mathcal{R}_{UEA}$ are disjoint.

The next obvious claim  shows that none of  regions  $R_1-R_{11}$ is empty.
The {\em double star} $S_{m,n}$, where $m,n \geq 2$, 
is the graph consisting of the union of two stars $K_{1,n}$ and $K_{1,m}$
 together with an edge  joining their centers.

{\bf Claim \ref{Rclassification}.1} 
\begin{itemize}
\item[(i)]  Any double star $S_{p,q}$ with $p,q \geq 3$,  is in $R_1$. 
\item[(ii)] The graph $G$ obtained from $S_{2,2}$ by subdividing once 
                    the edge joining the support vertises of $S_{2,2}$, is in $R_2$. 
\item[(iii)] The graph $G$ obtained from $K_4$ by adding a new vertex $v$,
                     joining it to three vertices of the $K_4$, and then subdividing once 
                     each of the edges incident to $v$, is in $R_3$. 
\item[(iv)] $C_6$ is in $R_4$. 
\item[(v)]  $K_{1,2}$ is in $R_5$. 
\item[(vi)] $K_{1,n}$, $n \geq 3$ is in $R_6$. 
\item[(vii)] The double star $S_{2,2}$   is in $R_7$.
\item[(viii)] $C_7$ is in $R_8$.  
\item[(ix)]   $C_4$ is in $R_9$.
\item[(x)]  The graph obtained from $2$ disjoint copies of $P_5$ by joining their central vertices is in $R_{10}$.
\item[(xi)]  $K_1 \cup K_{1,2}$  is in $R_{11}$.
 \end{itemize}
\end{proof}

\begin{lema} \cite{rv1} \label{er3}
Let   a graph $G$ have at least one edge. 
 Then $G$  is in   $\mathcal{R}_{CER}$  if and only if  $\Delta(G) \geq 2$ and 
$G$ is a forest  in which each component is an isolated vertex or a star of order at least $3$.
	\end{lema}

\begin{remark}\label{cearem}	
Using Lemma \ref{er3} it is easy to see that the following assertions hold.
\begin{itemize}
\item[(i)] 	A graph $G$ is in $R_5$ if and only if  $G = nK_{1,2}$, $n \geq 1$.
\item[(ii)]  A graph $G$ is in $R_6$ if and only if  each component of $G$ is  a star of order at least $4$.
\item[(iii)] A graph $G$ is in $R_{11}$ if and only if $\delta(G) = 0$ and 
                    each component of $G$ is an isolated vertex or a star of order at least $3$.
\end{itemize}
	\end{remark}

By Theorem \ref{Rclassification},  Claim \ref{Rclassification}.1}  and Remark \ref{cearem}	we immediately obtain:

\begin{corollary}\label{R11}
For connected graphs:  (a) the subset $R_{11}$ is empty, and  
(b) all $R_1, R_2,\dots,R_{10}$ are nonempty.
\end{corollary}

Now our  aim is to determine where trees of order at least $3$ fit into the subsets of the Venn diagram.

\begin{corollary}\label{trees}
For trees of order $n \geq 3$,  (a) all regions $R_3, R_4, R_8, R_9$ and $R_{11}$ 
of the Venn diagram (see Fig. \ref{fig:regions}) are empty,
 and (b) all  regions $R_1, R_2, R_5, R_6, R_7$ and $R_{10}$ are nonempty. 
\end{corollary}

\begin{proof}
Let $T$ be a tree. 
By Corollary \ref{R11},  $R_{11}$ is empty. 
Clearly $K_{1,2}$ is in $R_5$ and $K_{1,r}$, $r \geq 2$, is in $R_{6}$. 
Since a tree $T$ is  in $\mathcal{R}_{CVR}$  if and only if $T = K_2$ (see \cite{rhv}), 
$R_8$ and $R_9$ are empty.
  	Assume $T$ is in $\mathcal{R}_{UEA} \cap \mathcal{R}_{UER}$. 
	By Lemma \ref{ueaknown}, $V^-(T)$ is empty. Let $x$ be a leaf of $T$ and 	$\{y\} = N(x)$. 
	As $T$ is in $\mathcal{R}_{UER}$,  $\gamma_R(T) = \gamma_R(T-xy) = \gamma_R(T-x) + 1$, a contradiction. 
	Thus both $R_3$ and $R_4$ are empty.
	
The rest follows immediately by Theorem \ref{Rclassification}.  	
\end{proof}

Thus, we have shown that for trees of order $n \geq 3$, 
the regions of the Venn diagram can be reduced to the six shown in Fig. \ref{fig:trees}.

\begin{figure}[htbp]
	\centering
		\includegraphics{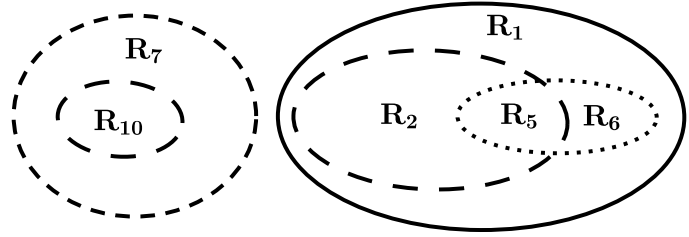}
	\caption{Regions of Venn diagram: trees}
	\label{fig:trees}
\end{figure}

A constructive characterization of the trees belonging to  $\mathcal{R}_{UEA}$ 
is given by Chellali  and  Jafari Rad \cite{cr},   and for the trees belonging to
 $\mathcal{R}_{UVR}$ - by the present author in \cite{samdmaa}.  
By Remark \ref{cearem},  all trees in $\mathcal{R}_{CER}$ are $K_{1,r}$, $r \geq 2$; 
hence $K_{1,2}$ is the unique element of  $R_5$,  
and $R_{6}$ consists of all stars $K_{1,r}$, $r \geq 3$. 

Let $U_i$ be the graph obtained by disjoint  copies  of $P_5$ and $P_{3+i}$ 
by joining the central vertex of $P_5$ with a central  vertex of $P_{3+i}$, $i = 1,2$.
Hansberg et al. \cite{rhv} show that 
$U_1$ and $U_2$ are the only trees which are in $\mathcal{R}_{CEA}$ (i.e. $R_{10}$).  

So, the following problem naturally arises.
\begin{problem}\label{R7}
Find a constructive characterization for trees in  $\mathcal{R}_{UER}$. 
\end{problem}

\begin{problem}\label{i-xii}
Let $\mu$ be a domination-related parameter. 
1) Give a characterization  of every of the twelve classes of graphs  stated in the introduction. 
2) Establish relationships among these twelve classes.
\end{problem}

This problem has been well-studied in the case when $\mu = \gamma$. 
 See the excellent article \cite{hh1} of Haynes and Henning and the references therein.

\section{The class $\mathcal{R}_{CVR}^k$} \label{cvrk}

We will write $\mathcal{R}_{CVR}^k$ instead of $k$-$VR^-_{\gamma_R}$, $k \geq 1$. 
Clearly  $\mathcal{R}_{CVR}^1 \equiv \mathcal{R}_{CVR}$. 
In \cite{r} Jafari Rad  posed the following question:
"What are the properties of graphs belonging to $\mathcal{R}_{CVR}^k$, $k \geq 3$?". 
Here we present some initial results on these  classes.  
We begin with easy observations. 

\begin{observation} \label{V1f} 
Let $G$ be a graph and $f$ a $\gamma_R$-function on  $G$ with 
$\emptyset \not= V_1^f \subsetneq V(G)$.    Then for each $S \subseteq V^f_1$, 
the function $f_S = (V_0^f; V_1^f-S; V_2^f)$ is a $\gamma_R$-function on $G-S$ and $\gamma_R(G-S) = \gamma_R(G)-|S|$. 
\end{observation}

\begin{observation} \label{n-1} 
Let $G$ be a graph of order $n$.
Then  $G$  is in $\mathcal{R}_{CVR}^k$ for each $k$ not less than $n-\gamma_R(G)  +1$. 
\end{observation}

There exist $n$-order graphs $G$ which are in $\mathcal{R}_{CVR}^k$ if and only if $k \geq n-\gamma_R(G)  +1$.  
For example all paths are such graphs (see Proposition \ref{path2}).

\begin{observation} \label{-k} 
Let a graph $G$ be in $\mathcal{R}_{CVR}^k$. 
Then $\gamma_R(G-v) \leq \gamma_R(G) + k-2$ for any vertex $v \in V(G)$.
  If $S \subsetneq V(G)$ and $|S|=k$, then 	$1\leq\gamma_R(G) -\gamma_R(G-S) \leq k$.   
\end{observation}

Let $v$ be the central vertex of a graph $G_k=K_{1,k}$, $k \geq 3$. 
Clearly $G_k$ is in $\mathcal{R}_{CVR}^k$ and  $\gamma_R(G_k-v) = \gamma_R(G_k) + k-2$.

\begin{theorem}\label{-2}
Let a graph $G$ be in $\mathcal{R}_{CVR}^k$ and $x \in V(G)$. 
\begin{itemize}
\item[(i)]    If $deg(x) \geq k-1$ and  $S$ is a set consisting of $x$ and $k-1$ of its neighbors,  
                  then $1 \leq\gamma_R(G) -\gamma_R(G-S) \leq 2$. 
\item[(ii)]    If $deg(x) = k$ then $\gamma_R(G) -\gamma_R(G-N(x)) =1$ and 
                    there is a $\gamma_R$-function $f_x$ on $G$ with 
                   $f_x(x) = 2$ and $f_x(u) = 0$ for all $u \in  N(x)$.  
\end{itemize}
\end{theorem}
\begin{proof}
(i) Since $G$ is in $\mathcal{R}_{CVR}^k$, the left side inequality is obvious. 
Let $f$ be a $\gamma_R$-function on $G-S$ and $g$ a $\gamma_R$-function on $\left\langle S \right\rangle$. 
Then a RD-function $h$ on $G$ defined as $h(u) = f(u)$ when $u \in V(G)-S$ and $h(u) = g(u)$ when $u \in S$ 
has weight $h(V(G)) = f(V(G-S)) + g(S) = \gamma_R(G-S)  + \gamma_R(\left\langle S \right\rangle) = 
 \gamma_R(G-S) +2$.

(ii) Let $f$ be a $\gamma_R$-function on $G-N(x)$. Clearly $f(x)=1$. 
       Define a RD-function $f_x$ on $G$ as follows: $f_x(x)=2$, $f_x(u) = 0$ when $u \in N(x)$, 
			and $f_x(v) = f(v)$ when $v \in V(G-N(x))$.  
			Since $f_x(V(G)) = f(V(G)-N(x)) +1 = \gamma_R(G-N(x)) + 1 \leq \gamma_R(G)-1+1 = \gamma_R(G)$, 
			$f_x$ is a $\gamma_R$-function  on $G$ and $\gamma_R(G) -\gamma_R(G-N(x)) =1$.
\end{proof}

\begin{corollary}\label{-2-1}
Let a graph $G$ be in $\mathcal{R}_{CVR}^k$ and $|V(G)| > k$. 
Let for every $S \subset V(G)$ with  $|S|=k$ 
is fulfilled $\gamma_R(G) -\gamma_R(G-S) \geq s$. 
\begin{itemize}
\item[(i)] (Jafari Rad \cite{r} when $k=s=2$)  If $s  \geq 2$ then $G$ has no vertex of degree $k$. 
\item[(ii)] If $s  \geq 3$ then  $\Delta(G) \leq k-2$. 
\end{itemize}
\end{corollary}

Now, our aim is to determine which classes $\mathcal{R}_{CVR}^k$ contain 
the paths $P_n$ and the cycles $C_m$, $2 \leq n$, $3 \leq m$.
We need some preparation for it. 
We let $x \equiv_3 y$ means $x \equiv y (mod \ 3)$.

\begin{observation}\label{int}
If $a$ and $b$ are nonnegative integers and $b \geq a$ then 
$\left\lceil  2a/3 \right\rceil + \left\lceil  2(b-a)/3 \right\rceil  \geq \left\lceil  2b/3 \right\rceil$. 
The equality holds if and only if there are integers $p$ and $q$ such that 
 $0 \leq p \leq q \leq 2$, $a \equiv_3 p$ and $b \equiv_3 q$. 
\end{observation}

Observation \ref{int}   will be used in the sequel without specific reference.

\begin{prop} [\cite{cdhh}]\label{o33} 
 $\gamma_{R} (C_n) = \left\lceil 2n/3 \right\rceil$ and $\gamma_{R} (P_m) = \left\lceil 2m/3 \right\rceil$. 
All cycles belonging to the class $\mathcal{R}_{CVR}$ are $C_{3k+1}$ and  $C_{3k+2}$, $k \geq 1$. 
\end{prop}

By $[020]^k$ we denote the sequence $0,2,0,..,0,2,0$ where $0,2,0$ is repeated $k$ times, $k \geq 0$;
if $k=0$ then the sequence is empty. 
Let $f$ be an RD-function on $P_n: x_0,x_1,..,x_{n-1}$ and let $0 \leq a \leq b \leq n$. 
Then by $(f(x_i))_{i=a}^b$ we denote the sequence  $f(x_a),f(x_{a+1}),..,f(x_b)$ of values of $f$.

\begin{proposition}\label{path}
For the path $P_n: x_0,x_1,..,x_{n-1}$ the following holds. 
\begin{enumerate}
\item[(i)] If $n \equiv_3 0$ then $\#\gamma_R(P_n) = 1$ and 
                   the only $\gamma_R$-function $f$ on $P_n$ is defined by 
									$(f(x_i))_{i=0}^{n-1} = [020]^{n/3}$.
\item[(ii)] If $n \equiv_3 1$ then $\#\gamma_R(P_n) =  (n+2)/3$ and 
                    any $\gamma_R$-function $f$ on $P_n$ can be defined by 
										$(f(x_i))_{i=0}^{n-1} = [020]^p 1  [020]^t$, 
                   for some $p,t \geq 0$ with $p+t =(n-1)/3$.
\item[(iii)]  If $n \equiv_3 2$ then $\#\gamma_R(P_n) = (n+4)(n+7)/18$ and  for any $\gamma_R$-function $f$ on $P_n$
                     either $(f(x_i))_{i=0}^{n-1} = [020]^s 1 [020]^p 1[020]^t$ for some $s,p,t \geq 0$ with $s+p+t =\left\lfloor n/3 \right\rfloor$ 
										or $(f(x_i))_{i=0}^{n-1} = 20[020]^{\left\lfloor n/3 \right\rfloor}$ or 
									 $(f(x_i))_{i=0}^{n-1} = [020]^k02[020]^l$ for some $k,l \geq 0$ with $k+l =\left\lfloor n/3 \right\rfloor$.
\end{enumerate}
\end{proposition}
\begin{proof}
First let $n \equiv_3 0$. By Proposition \ref{o33}, $\gamma_R(P_n) = 2n/3$. 
Suppose there is a $\gamma_R$-function $f$ on $P_n$ with $V_1^f \not = \emptyset$, 
 say $f(x_r) = 1$. Then the restriction of $f$ on $P_n-x_r$ is an RD-function and 
$2n/3 -1 \geq \gamma_R(P_n-x_r) = \left\lceil 2r/3\right\rceil + \left\lceil 2(n-r-1)/3\right\rceil  \geq  \left\lceil (2n-2)/3\right\rceil  =2n/3 $, 
a contradiction. Thus for any $\gamma_R$-function $f$ on $P_n$, $V_1^f  = \emptyset$.  
But then $V_2^f$ is a dominating set of $P_n$ of cardinality $n/3 = \gamma(G)$. 
Since $\{x_s \mid s\equiv_3 1\}$ is the unique $\gamma$-set of $P_n$, 
we deduce that  there is an unique $\gamma_R$-function $f$ on $P_n$ and  $(f(x_i))_{i=0}^{n-1} = [020]^{n/3}$.

Second let $n \equiv_3 1$.  By Proposition \ref{o33}, $\gamma_R(P_n) = (2n+1)/3$,
 which immediately implies $V_g^1 \not = \emptyset$ for 
each  $\gamma_R$-function $g$ on $P_n$; say $g(x_r) = 1$.
 Then the restriction $g^\prime$ of $g$ on $P_n-x_r$ is an RD-function and 
$(2n+1)/3 -1 \geq \gamma_R(P_n-x_r) = \left\lceil 2r/3\right\rceil + \left\lceil 2(n-r-1)/3\right\rceil  \geq 
 \left\lceil (2n-2)/3\right\rceil  =(2n-2)/3$. 
Hence $g^\prime$ is a $\gamma_R$-function on $P_n-x_r$ and 
$r\equiv_30$. 
Since each component of $P_n-x_r$ has order $\equiv_30$, $g^\prime$ is an unique $\gamma_R$-function on $P_n-x_r$. 
Thus there are exactly $(n+2)/3$ $\gamma_R$-functions on $P_n$, say $f_0, f_1,..,f_{(n-1)/3}$, 
defined by $(f_r(x_i))_{i=0}^{n-1} = [020]^r 1  [020]^{(n-1)/3-r}$. 

Finally let $n \equiv_3 2$.  By Proposition \ref{o33}, $\gamma_R(P_n) = (2n+2)/3$. 

{\it Case} 1: There is a $\gamma_R$-function $h$ on $P_n$ with $V_1^h \not = \emptyset$, say $h(x_r) = 1$. 
Hence the restriction of $h$ on $P_n-x_r$ is an RD-function and 
$(2n-1)/3 \geq \gamma_R(P_n-x_r) = \left\lceil 2r/3\right\rceil + \left\lceil 2(n-r-1)/3\right\rceil  \geq  \left\lceil (2n-2)/3\right\rceil  =(2n-1)/3$.
 But then  $r \not\equiv_32$ and  the restriction of $h$ on each component of $P_n-x_r$ is a $\gamma_R$-function. 
 Now, (a) if $r \equiv_30$ then $(h(x_i))_{i=0}^{n-1} = [020]^{r/3}1[020]^s1[020]^t$, 
where $r/3 +s+t = \left\lfloor n/3 \right\rfloor$, and 
(b)  if $r \equiv_31$ then $(h(x_i))_{i=0}^{n-1} = [020]^s1[020]^t1[020]^p$, 
where $s+t = \left\lfloor r/3 \right\rfloor$and $s+t+p = \left\lfloor n/3 \right\rfloor$. 
Thus, each $\gamma_R$-function $f$  on $P_n$ with $V_1^f \not= \emptyset$ has 
$(f(x_i))_{i=0}^{n-1} = [020]^a1[020]^b1[020]^c$,  where $0 \leq a,b,c$ and $a+b+c = \left\lfloor n/3 \right\rfloor$.   
Therefore the number of all such functions is $((n-2)/3+1)((n-2)/3+2)/2 = (n+1)(n+4)/18$. 

{\it Case} 2:  There is a $\gamma_R$-function $l$ on $P_n$ with $V_1^l  = \emptyset$. 
But then $V_2^l$ is an independent dominating set and there is at least one pair $x_r, x_{r+1}$ such that 
either $x_r \in V_2^f$ and $pn[x_r,V_2^f] = \{x_r,x_{r+1}\}$ or 
 $x_{r+1} \in V_2^f$ and $pn[x_{r+1},V_2^f ]= \{x_r,x_{r+1}\}$. 
Hence the restriction of $l$ on $P_n-\{x_r,x_{r+1}\}$ is an RD-function and 
$(2n+2)/3 -2 \geq \gamma_R(P_n-\{x_r,x_{r+1}\}) = 
\left\lceil 2r/3\right\rceil + \left\lceil 2(n-r-2)/3\right\rceil  \geq  \left\lceil (2n-4)/3\right\rceil  =(2n-4)/3$. 
But then  $r \equiv_30$ and  the restriction of $l$ on each component of $P_n-\{x_r,x_{r+1}\}$ is a $\gamma_R$-function. 
This implies that 
 $(l(x_i))_{i=0}^{n-1} = [020]^s20[020]^t$ or $(l(x_i))_{i=0}^{n-1} = [020]^s02[020]^t$, 
where $s = \left\lfloor r/3 \right\rfloor$ and $s+t=\left\lfloor n/3 \right\rfloor$. 
Thus for any $\gamma_R$-function $l$ on $P_n$ either $(l(x_i))_{i=0}^{n-1} = 20[020]^{\left\lfloor n/3 \right\rfloor}$ 
or  $(l(x_i))_{i=0}^{n-1} = [020]^s02[020]^t$, where $t+s = \left\lfloor n/3 \right\rfloor$. 
Therefore the number of all such functions is $\left\lfloor n/3 \right\rfloor + 2 = (n+4)/3$. 

Thus the number of all $\gamma_R$-functions on $P_n$ when $n\equiv_32$  is 
$(n+1)(n+4)/18 + (n+4)/3= (n+4)(n+7)/18$.
\end{proof}

\begin{corollary}\label{path1}
For the path $P_n: x_0,x_1,..,x_{n-1}$ the following holds. 
\begin{enumerate}
\item[(i)] If $n \equiv_3 0$ then  $V(P_n) = V_{R}^=(P_n)$.
\item[(ii)] If $n \equiv_3 1$ then $V(P_n) = V_{R}^-(P_n)  \cup V_{R}^=(P_n)$ and 
                    $V_{R}^-(P_n) = \{x_r \mid r \equiv_3 0 \}$. 
\item[(iii)]  If $n \equiv_3 2$ then $V(P_n) = V_{R}^-(P_n)  \cup V_{R}^=(P_n)$ and 
                    $V_{R}^=(P_n) = \{x_r \mid r \equiv_3 2 \}$. 
\end{enumerate}
\end{corollary}

\begin{proposition}\label{path2}
 $P_n \in \mathcal{R}^s_{CVR}$ if and only if $\left\lfloor n/3 \right\rfloor + 1 \leq s$. 
\end{proposition}
\begin{proof}
By Observation \ref{n-1}, $P_n \in \mathcal{R}^s_{CVR}$ for all $s \geq n - \gamma_R(P_n) +1$. 
Since $\gamma_R(P_n) = \left\lceil 2n/3\right\rceil$, $n - \gamma_R(P_n) +1 = \left\lfloor n/3\right\rfloor +1$.  
Let $P_n: x_0, x_1,..,x_{n-1}$ and $S = \{x_r \mid r \equiv_32\}$. 
Clearly $|S|  = \left\lfloor n/3\right\rfloor$ and (a) $P_n-S = (|S|+1)K_2$ when $n \equiv_32$, 
(b)   $P_n-S = |S|K_2 \cup K_1$ when $n \equiv_31$,  and 
(c)  $P_n-S = |S|K_2$ when $n \equiv_30$.  
In all cases,  $\gamma_R(P_n-S) = \left\lceil 2n/3\right\rceil =  \gamma_R(P_n)$. 
 By Corollary \ref{path1}, $\gamma_R(P_n-S_1)   \geq \gamma_R(P_n-S)$ for any $S_1 \subseteq S$. 
Thus, $P_n \not\in \mathcal{R}^s_{CVR}$ for all $s \leq \left\lfloor n/3\right\rfloor$. 
\end{proof}

Let  $f$ be a RD-function on  $C_n: x_0,x_2,..,x_{n-1},x_0$ and $0 \leq a \leq b$. 
 Then by $(f(x_i))_{i=a}^b$ we denote the sequence $f(x_a),f(x_{a+1}),..,f(x_b)$ of values of $f$, where
all the subscripts are taken modulo $n$.  

\begin{proposition}\label{cycle}
For the cycle $C_n$ the following holds. 
\begin{enumerate}
\item[(i)] If $n \equiv_3 0$ then $\#\gamma_R(C_n) = 3$ and  $f_0, f_1$ and $f_2$ 
                  defined by $(f_j(x_i))_{i=j}^{n-1+j} = [020]^{n/3}$, 
                  $j =0,1,2$ are all $\gamma_R$-functions on $C_n$.
\item[(ii)] If $n \equiv_3 1$ then $\#\gamma_R(C_n) =n$ and  any 
             $\gamma_R$-function $g$ on $C_n$ can be defined by 
             $(g(x_s))_{s=r}^{n+r-1} = 1[020]^{(n-1)/3}$, where $r \in \{0,1,..,n-1\}$.
\item[(iii)]  If $n \equiv_3 2$ then  $\#\gamma_R(C_n)=n(n+7)/6$ and  for any $\gamma_R$-function $f$ on $C_n$
                      either $(f(x_i))_{i=k}^{n-1+k} = [020]^ {\left\lfloor n/3 \right\rfloor}02$, $k \in \{ 0,1,..,n-1\}$  
											or $(f(x_i))_{i=s}^{n-1+s} = 1[020]^a1[020]^b$,  where $0 \leq s \leq s+3a \leq n-2$ and $a+b = \left\lfloor n/3 \right\rfloor$. 
\end{enumerate}
\end{proposition}
\begin{proof}
Assume first $n\equiv_30$. Then $\gamma_R(C_n) = 2n/3$ because of Proposition \ref{o33}. 
If $f(x_r) =1$ for some $\gamma_R$-function $f$ on $C_n$ 
then the restriction of $f$ on $C_n-x_r$ is an RD-function with 
weight $2n/3-1$. But $C_n-x_r$ is a path on $n-1$ vertices
 which implies $\gamma_R(C_n-x_r) = 2n/3$, a contradiction. 
Thus, for each $\gamma_R$-function $f$ on $C_n$, $V_1^f = \emptyset$. 
Then $V_2^f$ is a dominating set of $C_n$ and $|V_2^f| = n/3 = \gamma(G)$. 
Hence $V_2^f$ is a $\gamma$-set of $C_n$. But all $\gamma$-sets of $C_n$
 are $D_i = \{x_r \mid r\equiv_3i\}$, $i=0,1,2$. 
Thus either $(f(x_s))_{s=0}^{n-1} = [020]^{n/3}$ or 
$(f(x_s))_{s=1}^{n} = [020]^{n/3}$ or $(f(x_s))_{s=2}^{n+1} = [020]^{n/3}$. 

Second let $n\equiv_31$.  If $g$ is an arbitrary $\gamma_R$-function on $C_n$ then since 
$\gamma_R(C_n) = (2n+1)/3$, it follows that there is $x_r \in V_1^g$. 
By Proposition \ref{path}, the restriction of $g$ on $C_n-x_r$ is the unique $\gamma_R$-function on $C_n-x_r$ 
and  $(g(x_s))_{s=r+1}^{n+r-1} = [020]^{(n-1)/3}$. 
Thus for any $\gamma_R$-function $g$ on $C_n$, $(g(x_s))_{s=r}^{n+r-1} = 1[020]^{(n-1)/3}$, where $r \in \{0,1,..,n-1\}$. 
As a consequence, $\#\gamma_R(P_n) =n$.

Finally let $n\equiv_32$. Then $\gamma_R(C_n) =  (2n+2)/3$. 

{\it Case} 1: There is a $\gamma_R$-function $h$ on $C_n$ with $V_1^h \not = \emptyset$, say $h(x_r) = 1$. 
Then $C_n-x_r$ is a path on $n-1$ vertices and $\gamma_R(C_n-x_r) = (2n-1)/3 = \gamma_R(C_n) -1$. 
Hence the restriction of $h$ on $C_n-x_r$ is a $\gamma_R$-function. Now using Proposition \ref{path}, 
we obtain that  $(h(x_i))_{i=r}^{n+r-1} = 1[020]^{m}1[020]^t$, where $0 \leq m,t$ and $m+t = (n-2)/3$,  
or equivalently (to avoid repetitions), 
$(h(x_i))_{i=s}^{n-1+s} = 1[020]^a1[020]^b$,  where $0 \leq s \leq s+3a \leq n-2$ and $a+b = \left\lfloor n/3 \right\rfloor$. 
The number of all such functions is 
$\frac{n+1}{3} + 3(\frac{n+1}{3} -1) + 3(\frac{n+1}{3} - 2) + .. + 3.1 = \frac{n(n+1)}{6}$. 

{\it Case} 2: There is a $\gamma_R$-function $l$ on $C_n$ with $V_1^l = \emptyset$. 
Since $\gamma_R(C_n) = 2\gamma(C_n)$, $V_2^l$ is a $\gamma$-set of $C_n$. 
 But each $\gamma$-set $D$ of $C_n$ is independent and it is fixed by the placement of the only
vertex which is adjacent to two distinct elements of $D$. 
Thus, there are $n$ such functions, say $l_0, l_1,..,l_{n-1}$, where 
  $(l_k(x_i))_{i=k}^{n-1+k} = [020]^ {\left\lfloor n/3 \right\rfloor}02$, $k = 0,1,..,n-1$. 
	
	Thus the number of all $\gamma_R$-functions on $C_n$ when $n\equiv_32$  is 
	$ \frac{n(n+1)}{6} + n =  \frac{n(n+7)}{6}$.
\end{proof}

\begin{corollary}\label{cycle2}
If $n \not \equiv_3 0$ then  $V(C_n) = V_{R}^-(C_n)$.
If $n \equiv_3 0$ then  $V(C_n) = V_{R}^=(C_n)$. 
\end{corollary}

\begin{proposition}\label{cycle3}
If $n \not\equiv_3 0$ then  $C_n  \in \mathcal{R}^s_{CVR}$ for all $s \geq 1$. 
If $n \equiv_3 0$ then  $C_n  \in \mathcal{R}^s_{CVR}$   if and only if 
$n/3 + 1 \leq s$. 
\end{proposition}
\begin{proof}
First let $n \not\equiv_3 0$.  
Then $\gamma_R(P_{n-1}) < \gamma_R(C_n)$ (Proposition \ref{o33}). 
Corollary \ref{cycle2} implies $C_n \in \mathcal{R}^1_{CVR}$.  
By Corollary \ref{path1}, $\gamma_R(P_{n-1}-S) \leq \gamma_R(P_{n-1}) < \gamma_R(C_n)$
 for any $S \subseteq V(P_{n-1})$. Hence $C_n \in \mathcal{R}^s_{CVR}$ for all $s\geq 1$.

Now let $n \equiv_3 0$. Since  $\gamma_R(P_{n-1}) = \gamma_R(C_n)$ (Proposition \ref{o33}) 
and  $P_{n-1} \in \mathcal{R}^s_{CVR}$ if and only if  $s \geq \left\lfloor (n-1)/3 \right\rfloor + 1$
 (Proposition \ref{path2}),  it immediately follows that 
$C_n \in \mathcal{R}^k_{CVR}$ if and only if $k \geq \left\lfloor (n-1)/3 \right\rfloor + 2 =  n/3 + 1$. 
\end{proof}

\section{$\gamma_R$-graphs} \label{gammaR}

The idea of using the intersections of a family of sets to define
the adjacencies of a graph is so natural that it arose independently in
a number of areas in connection with both pure and applied problems
(see Roberts \cite{ro} and McKee and  McMorris \cite{mm}). 
Formally, for each integer $p \geq 1$, 
the $p$-{\em intersection graph} $\Omega_p(\mathcal{F})$ 
of the family $\mathcal{F} = \{S_1, S_2,..,S_n\}$ of subsets of a finite set $S$ 
is defined to be the graph $G$ having $V(G) =  \mathcal{F}$ with 
$S_iS_j \in E(G)$ if and only if $i \not = j$ and $|S_i \cap S_j| \geq p$. 
A graph $G$ is a $p$-{\em intersection graph} if there exists a family 
$\mathcal{F}$ such  that $G  \simeq \Omega_p(\mathcal{F})$.
The concept of the $p$-intersection graph was introduced by 
Jacobson, McMorris and Scheinerman \cite{jms}. 

For each generic invariant $\mu(G)$, a set $S$ having the desired property 
and cardinality  $\mu(G)$ is called a $\mu$-{\em set} of $G$. 
Denote by $\mathscr{D}_\mu(G)$ the family of all $\mu$-sets of $G$ 
and $\#\mu (G) = |\mathscr{D}_\mu(G)|$. 
Unfortunately, when $\mu$ is a  domination-related parameter,  
among all $\Omega_p(\mathscr{D}_\mu(G))$ graphs, 
 only the case   $\mu \equiv \gamma$ and $p = \gamma(G)-1$ is studied. 
The  ($\gamma(G)-1$)-intersection graph $\Omega_{\gamma(G)-1}(\mathscr{D}_\gamma(G))$
(under the name  $\gamma$-{\em graph of a graph})
 was first introduced  by Subramanian and  Sridharan \cite{ss} in 2008 
and was studied in  \cite{sar}, \cite{b} and  \cite{av}. 
There are at least $2$ important for investigation spanning subgraphs 
of  $\Omega_{\mu(G)-1}(\mathscr{D}_\mu(G))$. 
To define them we need to restrict the adjacency between $2$  elements
$D_i$ and $D_j$ of  $\mathscr{D}_\mu(G)$ as follows: 

\begin{itemize}
\item[$\mathcal{A}_1$:\ ] $D_i$ and $D_j$ are adjacent  if and only if  there are $v_i \in D_i$ and  $v_j \in D_j$
																											such that $v_iv_j \in E(G)$,  $D_i = D_j  - \{v_j\} \cup \{v_i\}$ and $D_j = D_i  - \{v_i\} \cup \{v_j\}$.     
\item[$\mathcal{A}_2$:\ ] $D_i$ and $D_j$ are adjacent  if and only if  there are $v_i \in D_i$ and  $v_j \in D_j$
																											such that $v_iv_j \not\in E(G)$,  $D_i = D_j  - \{v_j\} \cup \{v_i\}$ and $D_j = D_i  - \{v_i\} \cup \{v_j\}$.     
\end{itemize}

Define the graphs $\Omega_{\mu(G)-1}^a(\mathscr{D}_\mu(G))$ and 
$\Omega_{\mu(G)-1}^n(\mathscr{D}_\mu(G))$ as spanning subgraphs of 
$\Omega_{\mu(G)-1}(\mathscr{D}_\mu(G))$, where the adjacency is given by 
$\mathcal{A}_1$ and $\mathcal{A}_2$, respectively. 
In 2011, Fricke et al. \cite{fhhh}  began an  investigation on the graph
 $\Omega_{\gamma(G)-1}^a(\mathscr{D}_\gamma(G))$
 (under the name $\gamma$-{\em graph of a graph $G$}).  
For additional results on these  graphs, see \cite{chh} and \cite{e}.
Very recently, Mynhardt and  Teshima \cite{mt}, 
  present initial results  on the graphs $\Omega_{\mu(G)-1}^a(\mathscr{D}_\mu(G))$, 
	where $\mu$ is the locating-domination number, total-domination number, 
	paired-domination number, and the upper-domination number.

In this section we concentrate on the case when $\mu \equiv \gamma_R$. 
Formally, the $\gamma_R$-{\em graph of a graph} $G$,  is a graph which  vertex set
is $\mathscr{D}_R(G)$. 
 Adjacency between two  $\gamma_R$-functions $f$ and $g$ can be defined in several ways:

\begin{itemize}
\item[$\mathcal{P}_1$:\ ] $f$ and $g$ are adjacent if  there exist 
                                                   vertices  $u,v \in V(G)$ with $f(u)=2$ and $f(v)=0$ 
                                                    such that $g = ((V_0^f-\{v\}) \cup \{u\}; V_1^f; (V_2^f-\{u\}) \cup \{v\})$. 
																										
\item[$\mathcal{P}_2$:\ ]  $f$ and $g$ are adjacent if  there exist vertices  $u,v \in V(G)$
                                                    with $f(u)=2$ and $f(v)=1$ 
                                                    such that  $g = (V_0^f; (V_1^f-\{v\}) \cup \{u\}; (V_2^f-\{u\}) \cup \{v\})$. 
																										
\item[$\mathcal{P}_3$:\ ]  $f$ and $g$ are adjacent if  
                                                     (a) there exist vertices  $u,v \in V(G)$ with $f(u)=f(v)=1$, and 
																										  (b) $g(z) = f(z)$ for all $z \in V(G) - \{u,v\}$  and $\{g(u), g(v)\} = \{0,2\}$
																											 (clearly $u$ and $v$ are adjacent).  
																											
\item[$\mathcal{P}_4$:\ ]  $f$ and $g$ are adjacent if  there exists a triangle 
                                                      $u,v,w,u$  with $f(u)=f(v)=1$, and																								
																		$g = ((V_0^f - \{w\}) \cup \{u,v\}; V_1^f-\{u,v\}; V_2^f \cup \{w\})$. 
																											
\item[$\mathcal{P}_{1^a}$:\ ] $f$ and $g$ are adjacent if  there exist 
                                                   vertices  $u,v \in V(G)$ with $f(u)=2$ and $f(v)=0$ 
                                                    such that $g = ((V_0^f-\{v\}) \cup \{u\}; V_1^f; (V_2^f-\{u\}) \cup \{v\})$ and $uv \in E(G)$.  
																										
\item[$\mathcal{P}_{1^n}$:\ ] $f$ and $g$ are adjacent if  there exist 
                                                   vertices  $u,v \in V(G)$ with $f(u)=2$ and $f(v)=0$ 
                                                    such that $g = ((V_0^f-\{v\}) \cup \{u\}; V_1^f; (V_2^f-\{u\}) \cup \{v\})$ and $uv \not\in E(G)$.  																										
\end{itemize}
\medskip

Let us say that an adjacency  has the following:

\begin{itemize}
\item[(i)]  type \  {\bf i}    if  it   is given by  $\mathcal{P}_i$, $i \in \{ 1,1^a,1^n,2,3,4\}$.

\item[(ii)]  type \  {\bf ij}    if it   is given by  $\mathcal{P}_i   \vee  \mathcal{P}_j$, where 
																								$i \in \{ 1,1^a,1^n,2,3\}$, $j\in \{2,3,4\}$ and $i \not \equiv j$. 
\item[(iii)]  type \  {\bf ijk}    if it  is given by  $\mathcal{P}_i   \vee  \mathcal{P}_j  \vee  \mathcal{P}_k$, 
                       where  $i \in \{ 1,1^a,1^n,2\}$, $j \in \{ 2,3\}$,  $k \in \{3,4\}$, and  $i \not= j \not=k$. 
											
\item[(iv)]  type \  {\bf i234}    if it   is given by  
                       $\mathcal{P}_i   \vee  \mathcal{P}_2  \vee  \mathcal{P}_3 \vee  \mathcal{P}_4$, 
											 where  $i \in \{ 1,1^a,1^n\}$.									
\end{itemize}
\medskip

Let $\mathfrak{T}$  be the set of all  above defined  types of   adjacency and $\mathfrak{t} \in \mathfrak{T}$.
The $\mathfrak{t}$-$\gamma_R$-{\em graph of a graph} $G$,
denoted by  $G(\mathfrak{t}, \gamma_R)$, 
is a $\gamma_R$-graph of $G$ with adjacency having type $\mathfrak{t}$.
  We say that $\gamma_R$-functions $f$ and $h$ on $G$ are $\mathfrak{t}$-{\em adjacent} 
whenever $fg$ is an edge of $G(\mathfrak{t}, \gamma_R)$.
 Clearly $G(\mathfrak{t}, \gamma_R)$ is a 
spanning subgraph of $G({\bf 1234}, \gamma_R)$.

The {\em cartesian product} of graphs $G_1$ and $G_2$, denoted by $G_1\square G_2$
 has vertex set  $V (G_1\square G_2) = V (G_1) \times V (G_2)$ and two vertices $(u_1, u_2)$
 and $(v_1, v_2)$ are adjacent if $u_1 = v_1$ and $u_2v_2 \in  E(G_2)$ or
$u_1v_1 \in E(G_1)$ and $u_2 = v_2$.

\begin{theorem}\label{cart}
Let $G_1$ and $G_2$ be disjoint graphs and $G_0 = G_1 \cup G_2$. 
Then $G_0(\mathfrak{t}, \gamma_R)$ and $ G_1(\mathfrak{t}, \gamma_R) \square G_2(\mathfrak{t}, \gamma_R)$ 
are isomorphic, where $\mathfrak{t} \in \mathfrak{T}$.
\end{theorem}
\begin{proof}
If $f$ is a $\gamma_R$-function on $G_0$ then clearly $f|_{G_i}$ is a $\gamma_R$-function on $G_i$, $i=1,2$. 
Also, if $h_i$ is a $\gamma_R$-function on $G_i$, $i=1,2$, 
then a RD-function $h$ on $G_0$ defined as $h(u) = h_i(u)$ when $u \in V(G_i)$, $i=1,2$, 
is a $\gamma_R$-function on $G_0$. Thus, there is a bijection  
$\varphi: V( G_0(\mathfrak{t}, \gamma_R) ) \rightarrow V( G_1(\mathfrak{t}, \gamma_R)) \times  V(G_2(\mathfrak{t}, \gamma_R))$ 
given by $\varphi(f)  = (f|_{G_1}, f|_{G_2})$. 

Let $fh \in E( G_0(\mathfrak{t}, \gamma_R) )$. 
By the definition of the adjacency in $ G_i(\mathfrak{t}, \gamma_R)$, 
$i=0,1,2$,  we have that either $f|_{G_1}h|_{G_1}$ is an edge of $ G_1(\mathfrak{t}, \gamma_R)$ and 
$f|_{G_2} = h|_{G_2}$, or $f|_{G_1} = h|_{G_1}$ and $f|_{G_2}h|_{G_2}$ is an edge of $ G_2(\mathfrak{t}, \gamma_R)$. 
Thus $\varphi(f)\varphi(h)$ is an edge of $ G_1(\mathfrak{t}, \gamma_R) \square G_2(\mathfrak{t}, \gamma_R)$.  
Finally, let $(f_1, f_2)(h_1, h_2) \in E( G_1(\mathfrak{t}, \gamma_R) \square G_2(\mathfrak{t}, \gamma_R))$. 
The adjacency in $ G_i(\mathfrak{t}, \gamma_R)$, $i=0,1,2$, 
now implies $\varphi^{-1}((f_1, f_2))\varphi^{-1}((h_1, h_2))$ is an edge of  $G_0(\mathfrak{t}, \gamma_R)$. 
\end{proof}

\begin{example}\label{Pm}
By Proposition \ref{path} (i)-(ii), it immediately follows that $P_n({\bf 1^a}, \gamma_R)$
 is edgeless when $n \not\equiv_3 2$.  So, let $n \equiv_3 2$. 
By Proposition \ref{path} (iii) and its proof, we have that  $P_n({\bf 1^a}, \gamma_R)$ 
has  (a) $r=(n+1)(n+4)/18$ vertices of the form $[020]^s1[020]^p1[020]^t$,  
(b) $(n+1)/3$ vertices of the form $[020]^k02[020]^l$, and (c) the vertex $20[020]^{\left\lfloor n/3 \right\rfloor}$. 
 All vertices from (a) are isolated in $P_n({\bf 1^a}, \gamma_R)$, while the rest vertices induced the path:

$20[020]^s, [020]20[020]^{s-1},  [020]^220[020]^{s-2}, \dots, [020]^s20,  [020]^s02$, 
 \\
where $s = \left\lfloor n/3 \right\rfloor$. Thus $P_n({\bf 1^a}, \gamma_R) = \overline{K_r} \cup P_{s+2}$. 
\end{example}

\begin{example}\label{Cn}
Proposition \ref{cycle} (i)-(ii) implies $C_n({\bf 1^a}, \gamma_R) = \overline{K_3}$
 when  $n \equiv_3 0$ and $C_n({\bf 1^a}, \gamma_R) = \overline{K_n}$ when $n \equiv_3 1$. 
It remains the case $n \equiv_3 2$. 
By the proof of Proposition \ref{cycle} (iii), 
all $n(n+1)/6$ $\gamma_R$-functions $f_i$ on $C_n$ with $V_1^{f_i} \not= \emptyset$ are 
isolated vertices of $C_n({\bf 1^a}, \gamma_R)$. 
The rest $n$ $\gamma_R$-functions on $C_n$ are 
  $(l_k(x_i))_{i=k}^{n-1+k} = [020]^ {\left\lfloor n/3 \right\rfloor}02$, $k = 0,1,..,n-1$. 
	Clearly all neighbors of $l_k$ in $C_n({\bf 1^a}, \gamma_R)$ are $l_{k-3}$ and $l_{k+3}$. 
	It remains to note that $l_0,l_3,..,l_{3k}, l_{1}, l_4,.., l_{3k+1}, l_2,l_5,..,l_{3k+2}=l_0$ 
	is a cycle, where $n=3k+2$. Thus $C_n({\bf 1^a}, \gamma_R) = \overline{K_a} \cup C_n$, 
	where $a=n(n+1)/6$.  
	\end{example}

We leave the description of the graphs $P_m(\mathfrak{t}, \gamma_R)$ and 
$C_n(\mathfrak{t}, \gamma_R)$ when  $\mathfrak{t} \in \mathfrak{T} - \{{\bf1^a}\}$ to the reader.

\begin{proposition}\label{defin}
Let $G$ be a graph, $f$ a $\gamma_R$-function on $G$,  $u,v \in V(G)$ and $f(u)=2$. 
\begin{itemize}
\item[(i)]  Let $f(v) = 0$ and $f_1 = ((V_0^f-\{v\}) \cup \{u\}; V_1^f; (V_2^f-\{u\}) \cup \{v\})$ 
                     a RD-function on $G$.
										Then $f_1$ is a $\gamma_R$-function on $G$,  $dist(u,v) \leq 2$ and $|pn[u,V_2^f]| \geq 2$. 
										If $dist(u,v) = 2$ then $pn[u,V_2^f] \subseteq N(v)$. 
										If $uv \in E(G)$ then $pn[u,V_2^f] \subseteq N[v]$. 
										If $G$ is a tree then $uv \in E(G)$,  $pn[u,V_2^f] = \{u,v\}$, and 
										$f_2 = (V_0^f-\{v\}; V_1^f \cup \{u,v\}; V_2^f-\{u\})$  is a $\gamma_R$-function on $G$.
										
\item[(ii)] Let $f(x) = 1$ and $h = (V_0^f; (V_1^f-\{x\}) \cup \{u\}; (V_2^f-\{u\}) \cup \{x\})$ 
                   a RD-function on $G$. Then (a) $G$  has a cycle and 
										(b) $h$ is a $\gamma_R$-function on $G$, $dist(u,x) = 2$, $u \in pn[u,V_2^f]$, 
										     $pn[u,V_2^f] - \{u\} \subseteq N(x)$ and $|pn[u,V_2^f]| \geq 3$.			
\end{itemize}
\end{proposition}

\begin{proof}
First note that for each $\gamma_R$-function $l$ on $G$, 
if $x \in V_2^l$ then $|pn[x,V_2^l]| \geq 2$. 
Clearly $f_1$ and $h$ are $\gamma_R$-functions on $G$. 

(i) A vertex $v$ has to dominate all vertices in $pn[u,V_2^f]$. 
Hence $dist(u,v) \leq 2$ and if the equality holds, then 
 $pn[u,V_2^f] \subseteq N(v)$. Now let $uv \in E(G)$. 
Then clearly $pn[u,V_2^f] \subseteq N[v]$. Hence  if $G$ is a tree, then 
$pn[u,V_2^f]  = \{u,v\}$. But then $f_1 = (V_0^f-\{v\}; V_1^f \cup \{u,v\}; V_2^f-\{u\})$
 is a $\gamma_R$-function on $G$. 

(ii)  By Observation \ref{o2}, $dist(u,v) \geq 2$ and $u \in pn[u,V_2^f]$. 
Since $v$ has to dominate all vertices in $pn[u,V_2^f]-\{u\}$, 
$dist(u,v) = 2$ and $pn[u,V_2^f] - \{u\} \subseteq N(v)$. 
Suppose $pn[u,V_2^f] = \{u,w\}$. Then $wv \in E(G)$ and $f(w)=0$. 
Now $l_1 = (V_1^f-\{w\}; V_1^f \cup \{w,u\}; V_2^f - \{u\})$ is a $\gamma_R$-function 
on $G$ with $l_1(u) = l_1(w) = l_1(v) = 1$, a contradiction with Observation \ref{o2}. 
Thus $|pn[u,V_2^f]| \geq 3$ and since $pn[u,V_2^f]-\{u\} \subseteq N(v)$, 
$G$ has a cycle. 
\end{proof}

The next three corollaries are an immediate consequence of 
Proposition \ref{defin}.

\begin{corollary}\label{classescor}
If $G$  is a tree of order at least $3$, then all the following hold: 
\begin{itemize}
\item[(a)] $G(\mathfrak{t}, \gamma_R)$ is edgeless for all 
       $\mathfrak{t} \in \{\bf 1^n, 2, 4, 1^n2, 1^n4, 24, 1^n24\}$, 
\item[(b)] $G({\bf 1^a}, \gamma_R) \equiv  G(\mathfrak{t}, \gamma_R)$
  for all $\mathfrak{t} \in \{\bf 1, 1^a2, 1^a4, 12, 14, 1^a24, 124\}$, 
\item[(c)]  $G({\bf 3}, \gamma_R) \equiv  G(\mathfrak{t}, \gamma_R)$
  for all $\mathfrak{t} \in \{\bf 1^n3, 23, 34, 1^n23, 1^n34, 234, 1^n234\}$, 
\item[(d)]  $G({\bf 1^a3}, \gamma_R) \equiv  G(\mathfrak{t}, \gamma_R)$
  for all $\mathfrak{t} \in \{\bf 13, 1^a23, 1^a34, 123, 134, 1^a234, 1234\}$.
	\end{itemize}
\end{corollary}

\begin{corollary}\label{1a1a3}
If $G$  is a tree and $f,f_1 \in \mathscr{D}_{\gamma_R}(G)$ are ${\bf 1^a}$-adjacent,   
then there is $f_2 \in \mathscr{D}_{\gamma_R}(G)$ which is 
${\bf 3}$-adjacent to both $f$ and $f_1$. 
\end{corollary}

\begin{corollary}\label{treescor}
Let $G$ be a tree, $f$ a $\gamma_R$-function on  $G$, $v \in V_0^f$ and $x \in V_2^f$.
Then (a) there is at most one vertex  $u \in V_2^f$ such that $f_{vu} = ((V_0^f-\{v\}) \cup \{u\}; V_1^f; (V_2^f-\{u\}) \cup \{v\})$
 is also a $\gamma_R$-function on  $G$, and 
(b) there is at most one vertex  $y \in V_0^f$ such that $f_{xy} = ((V_0^f-\{y\}) \cup \{x\}; V_1^f; (V_2^f-\{x\}) \cup \{y\})$
 is also a $\gamma_R$-function on  $G$.
\end{corollary}

\begin{proposition}\label{Delta}
Let $f$ be a $\gamma_R$-function on a tree $G$. 
Then  $deg(f, G({\bf 1^a}, \gamma_R)) \leq |V_2^f|$. In particular, 
$\Delta(G({\bf 1^a}, \gamma_R)) \leq \max \{|V_2^h| \mid h \in \mathscr{D}_R(G)\} \leq \gamma_R(G)/2$.
\end{proposition}

\begin{proof} 
The inequality  $deg(f, G({\bf 1^a}, \gamma_R)) \leq |V_2^f|$ is true 
because of Corollary \ref{treescor}. The rest is obvious.
\end{proof}

Let $T_n$ be the tree obtained by $K_{1,n}$, $n \geq 2$, 
by subdividing twice all the edges of $K_{1,n}$ except one. 
Let $S$ be the set of all support vertices of $T_n$. 
Then $f=(V(T_n)-S; \emptyset; S)$ is a $\gamma_R$-function on $T_n$ 
and $\Delta(T_n({\bf 1^a}, \gamma_R)) = \gamma_R(T_n)/2 = |S|$.
Thus the bounds stated in Proposition \ref{Delta} are attainable.

\begin{theorem}\label{treesexists}
If a graph $G({\bf 1^a}, \gamma_R)$ has a triangle, then a graph $G$ also has a triangle.
\end{theorem}
\begin{proof}
Suppose, $f,g,h \in V(G({\bf 1^a}, \gamma_R))$ induce a triangle. 
Then there are vertices $a,b,c,d,m,n \in V(G)$ such that: $ab, cd, mn \in E(G)$, 
$g = ((V_0^f-\{a\})\cup \{b\}; V_1^f; (V_2^f-\{b\}) \cup \{a\})$, 
$h =  ((V_0^g-\{c\})\cup \{d\}; V_1^g; (V_2^g-\{d\}) \cup \{c\}) = 
       ((((V_0^f-\{a\})\cup \{b\}) - \{c\}) \cup \{d\}; V_1^f; (((V_2^f-\{b\}) \cup \{a\}) - \{d\}) \cup \{c\})$, 
	and 
	$h = ((V_0^f-\{m\})\cup \{n\}; V_1^f; (V_2^f-\{n\}) \cup \{m\})$.				

Hence $(V_0^f-\{m\})\cup \{n\} = (V_0^f-\{a,c\})\cup \{b,d\}$, 
which implies $m \in  \{a,c\}$, $n \in \{b,d\}$ and $|\{a,c\} \cap \{b,d\}|=1$.
If $a=d$ then $c=m$, $b=n$ and $C_3: a,b,c,a$ is a triangle in $G$. 
If $b=c$ then $a=m$, $d=n$ and  $C_3: a,b,d,a$ is a triangle in $G$. 
\end{proof}

\begin{corollary}\label{triangle}
If a graph $G$ has no triangle, then $G({\bf 1^a}, \gamma_R)$ also has no triangle. 
In particular, if $G$ is a tree then $G({\bf 1^a}, \gamma_R)$ is a forest.
\end{corollary}

Let $H$ be an arbitrary graph with vertex set $V(H) = \{v_1, v_2,..,v_n\}$, 
 $H_1 \simeq K_{1,3}$, and  $H_2 \simeq 2K_2$.    
Construct a new graph $F_H$ formed from the union of $H, H_1$ and $H_2$ by 
 connecting each vertex of $H$ to the central vertex $v$ of $H_1$,  and to each vertex of $H_2$. 
Clearly all $\gamma_R$-functions are $f_i = (V(F_H) - \{v_i,v\}; \emptyset; \{v_i,v\})$, $i= 1,2,..,n$. 
This immediately implies that for any $\gamma_R$-functions $f_i$ and $f_j$ on $F_H$, 
 the following holds: (a) $f_i$ and $f_j$ are always ${\bf 1}$-adjacent, 
(b)  $f_i$ and $f_j$ are ${\bf 1^a}$-adjacent if and only if $v_iv_j \in E(F_H)$, and 
(c) $f_i$ and $f_j$ are ${\bf 1^n}$-adjacent if and only if $v_iv_j \in E(\overline{F_H})$.
But then   $F_H({\bf 1}, \gamma_R) \simeq K_n$, $F_H({\bf 1^a}, \gamma_R) \simeq H$  
and $F_H({\bf 1^n}, \gamma_R) \simeq \overline{H}$.  
Hence $F_{\overline{H}}({\bf 1^n}, \gamma_R) \simeq \overline{\overline{H}} = H$.
Thus, the following is valid.

\begin{theorem}\label{11a1n}
For any graph $H$, there exist  graphs $G^{\bf a}$ and  $G^{\bf n}$ such that 
$H  \simeq G^{\bf a}({\bf 1^a}, \gamma_R)$ and $H  \simeq G^{\bf n}({\bf 1^n}, \gamma_R)$. 
\end{theorem}

\begin{theorem}\label{everytree}
 For each  tree $T$ there is a graph $G$ such that $T \simeq G({\bf 2}, \gamma_R)$. 
\end{theorem}
\begin{proof}
We proceed by induction on the order $n$ of a tree $T$.
It is easy to  see that  $K_1 = K_1({\bf 2}, \gamma_R)$ and $K_2 = K_{2,3}({\bf 2}, \gamma_R)$.  
 Assume now that the theorem is true for all trees $T$ of order at most $n$, 
and  let $T_{n+1}$ be a tree with $V(T_{n+1}) = \{v_1,v_2,..,v_{n+1}\}$.  
Let $v_{n+1}$ be a leaf,  $v_n$ the neighbor of $v_{n+1}$ and $T_n = T_{n+1}-v_{n+1}$.
By induction hypothesis,  there is a graph $G_n$ with  $T_n \simeq G_n({\bf 2}, \gamma_R)$.  
Denote by  $f_i$  the $\gamma_R$-function on $G_n$ corresponding to $v_i \in V(T_n)$, $i=1,2,..,n$, 
and let   $V_2^{f_n} = \{u_1,u_2,..,u_k\}$.    

 Construct a new graph $G_n^1$as follows: 
(a) attach to $u_i$ three leaves $u_{i1}$, $u_{i2}$ and $u_{i3}$, $i = 1,..,k$, 
(b)   add a new vertex $x$ and join it to each element of $A = \cup_{i=1}^k\{u_{i1}, u_{i2}, u_{i3}\}$. 
Finally, let $G_{n+1}$ be the graph obtained from the union of $G_n^1$ and $K_{2,3}$ 
by identifying $x$ and a vertex of degree $3$ in $K_{2,3}$ in a vertex labeled $x$. 
Let $w$ be the second $3$-degree vertex of $K_{2,3}$. 

Clearly $\gamma_R(T_{n+1}) = \gamma_R(T_n) +3$ and 
all $\gamma_R$-functions on $T_{n+1}$ are: 
$l_i = (V_0^{f_i} \cup A; V_1^{f_i} \cup \{w\}; V_2^{f_i} \cup \{x\})$ where $i=1,2,..,n$, 
and $l_{n+1} = (V_0^{f_{n}} \cup A; V_1^{f_n} \cup \{x\}; V_2^{f_n} \cup \{w\})$. 
Then there is an isomorphism between $G_{n+1}({\bf 2}, \gamma_R)  - l_{n+1}$ and  $T_n$
given by   $l_i \rightarrow v_i$, $i=1,2,..,n$. 
Since $l_n$ is the only neighbor of $l_{n+1}$ in $G_{n+1}({\bf 2}, \gamma_R)$, we obtain  
   $T_{n+1}  \simeq 	G_{n+1}({\bf 2}, \gamma_R)$. 		
	\end{proof}

\begin{problem}\label{gammar2}
Give a constructive characterization of all trees $T$ for which 
 $\Delta(T({\bf 1^a}, \gamma_R))  =  \gamma_R(G)/2$ (Proposition \ref{Delta}). 
\end{problem}

\begin{problem}\label{omega}
Let $\mu$ be a domination-related parameter and $p$ a positive integer. 
              Find results on the graphs  $\Omega_{p}(\mathscr{D}_\mu(G))$, 
               $\Omega_{\mu(G)-1}(\mathscr{D}_\mu(G))$, $\Omega_{\mu(G)-1}^a(\mathscr{D}_\mu(G))$ and 
              $\Omega_{\mu(G)-1}^n(\mathscr{D}_\mu(G))$, where $\mathscr{D}_\mu(G)$ is the set of all 
							$\mu$-sets of $G$. 
\end{problem}

\begin{problem}\label{Romanmu}
Define appropriate adjacencies between the elements of $\mathscr{D}_\mu(G)$,  
where $\mu$ is at least one of the independent/connected/signed/total/double Roman domination number
and signed/minus domination number. Investigate the obtained $\mu$-graphs. 
\end{problem}

\bigskip


\begin{thebibliography}{99}



\bibitem{av} S. Aparna Lakshmanan, A. Vijayakumar, 
The gamma graph of a graph, 
 AKCE J. Graphs. Combin., 7, No. 1 (2010), 53--59



\bibitem{b} A.Bien,   Gamma graph of some special classes of trees
Annales Mathematicae Silesianae 29 (2015), 25–34



\bibitem{cr} M. Chellali, N. Jafari Rad, 
                        Roman domination stable graphs upon edge-addition, 
												Util. Math. 96, 165--178 (2015). 
												
 \bibitem{crv} M. Chellali, N. Jafari Rad, and L. Volkmann, 
	Some results on Roman domination edge critical graphs, 
	AKCE Int. J. Graphs Comb., 9, No. 2 (2012),  195--203.												
                    
\bibitem{cdhh} 
E. J. Cockayne, P. A. Dreyer, Jr., S. M. Hedetniemi and S. T. Hedetniemi,
  Roman domination in graphs. {\it Discrete Math.} {\bf 278} (1--3) (2004) 11--22.
	
\bibitem{chh} E. Connelly, S.T. Hedetniemi, K.R. Hutson, A note on $\gamma$-Graphs,
                    AKCE Inter. J.Graphs Comb. 8(1) (2010), 23--31.	
										
\bibitem{e} M. Edwards, Vertex-critically and bicritically for independent domination and total
domination in graphs. PhD Dissertation, University of Victoria, 2015.										
 													
\bibitem{fhhh} Gerd H. Fricke, Sandra M. Hedetniemi, Stephen T. Hedetniemi, Kevin R. Hutson, 
$\gamma$-graphs of graphs, 
Discuss.  Math. Graph Theory 31 (2011) 517--531


																							
\bibitem{hr} M. Hajian,  N. Jafari Rad, 
                         On the Roman domination stable graphs, 		
                         Discus. Math. Graph Theory xx (xxxx) 1--13, doi:10.7151/dmgt.1975
																	
			
		\bibitem{rhv0} A. Hansberg,  N.J. Rad , L. Volkmann,
		Characterization of Roman domination critical unicyclic graphs,
																Utilitas mathematica,  86(2011),  129--146
																										
	\bibitem{rhv} A. Hansberg,  N.J. Rad , L. Volkmann,
                          Vertex and edge critical Roman domination in graphs,
													Utilitas Mathematica 92(2013), 73--97. 															
															
\bibitem {hhs1} T.W. Haynes, S.T. Hedetniemi, P.J. Slater, {\em Fundamentals of Domination in Graphs}, Marcel Dekker, New York, 1998.


\bibitem{hh1}  T.W. Haynes, M.A. Henning, Changing and unchanging domination: a classification, 
                               Discr. Math., 272 (2003) 65--79											
															
\bibitem{jms} M.S. Jacobson, F.R. McMorris, E.R. Scheinerman, 
                           General Results on Tolerance Intersection Graphs,
													J. Graph Theory, Vol. 15, No. 6, 573--577 (1991)

												
\bibitem{mm} T. A. McKee,  F. R. McMorris, Topics in Intersection Graph Theory,
                            1999, SIAM, Philadelphia, 
												
\bibitem{mt} C.M. Mynhardt, L.E. Teshima, A note on some variations of the $\gamma$-graph, 
                          manuscript, https://arxiv.org/abs/1707.02039

\bibitem{r} N. J. Rad, A generalization of Roman domination critical graphs, JCMCC 83 (2012), 33--49.

\bibitem{rv1} N. J. Rad and L. Volkmann, Changing and unchanging the Roman
                       domination number of a graph, Util. Math. 89 (2012) 79-95
													
\bibitem{re1} ReVelle CS (1997a) Can you protect the Roman Empire? Johns Hopkins Mag 49(2):40

\bibitem{re2} ReVelle CS (1997b) Test your solution to ?an you protect the Roman Empire? Johns Hopkins Mag
49(3):70

\bibitem{rer} ReVelle CS, Rosing KE (2000) Defendens Imperium Romanum: a classical problem in military. Am Math
Mon 107(7):585--94

\bibitem{ro}  F. Roberts, Discrete Mathematical Models. Prentice-Hall, Englewood Cliffs (1976).


\bibitem{samdmaa} V. Samodivkin, Roman domination in graphs: The class $R_{UVR}$, 
                                     Discr. Math. Algorithms and Appl., Vol. 8, No. 3 (2016) 1650049 (14 pages) 
																		
\bibitem{ss} K. Subramanian, N. Sridharan, $\gamma$-graph of a graph, 
                         Bulletin of kerala mathematics association, 5(1)(2008), 17--34
		
																														
	\bibitem{sar} N. Sridharan,  S. Amutha, S. B. Rao, 
			                           Induced subgraphs of gamma graphs, 
	                            Discr. Math. Algorith.  Appl. Vol. 5, No. 3 (2013) 1350012 (5 pages)
																											

\end{thebibliography}
\end{document}